\documentclass[11pt]{article}
\usepackage{amsfonts, amsmath, amssymb, amsthm, enumitem, etoolbox, wasysym, graphicx, dirtytalk, titlesec}
\usepackage[backend=biber, style=numeric , giveninits=true, doi=false,isbn=false,url=false, sorting=nyt]{biblatex}

\usepackage[margin=1.0in]{geometry}
\usepackage{footnote}

\theoremstyle{definition}
\newtheorem{thm}{Theorem}[section]
\newtheorem{lem}[thm]{Lemma}
\newtheorem{prop}[thm]{Proposition}
\newtheorem{cor}[thm]{Corollary}

\newtheorem*{conj}{Conjecture}

\titleformat{\section}
  {\center\bf}{\thesection.}{0.5em}{}
\renewcommand{\O}{\mathcal{O}}
\newcommand{\p}{\mathfrak{p}}
\newcommand{\q}{\mathfrak{q}}
\renewcommand{\a}{\alpha}
\renewcommand{\b}{\mathfrak{b}}
\newcommand{\eps}{\epsilon}
\newcommand{\fun}{\varepsilon}
\newcommand{\Q}{\mathbb{Q}}
\newcommand{\supp}{\mathcal{Q}}
\newcommand{\Z}{\mathbb{Z}}

\newcommand{\ord}{\operatorname{ord}}
\renewcommand{\mod}[1]{\left(\bmod{\,#1}\right)}

\begin{document}
\begin{center}
\uppercase{\textbf{Wieferich primes in number fields and the conjectures of Ankeny--Artin--Chowla and Mordell}}
\vskip 10pt
{\textsc{Nic Fellini and M. Ram Murty}}
\end{center}
\vskip 10pt

\centerline{\small\bf Abstract}
\noindent
{\small
For a prime $p\equiv 1 \mod{4}$, let
\[
\fun = \frac{1}{2}\left( t + u\sqrt{p}\right) 
\]
be the fundamental unit of the real quadratic field $\Q(\sqrt{p})$.
In 1951, N. Ankeny, E. Artin, and S. Chowla asked whether  $p$ can divide $u$.  They suggested that this can never
happen and this has since been called the Ankeny--Artin--Chowla (AAC) conjecture.   
 We show that if $\p$ is the prime above $p$ in $\Q(\sqrt{p})$, then the AAC conjecture is false if and only if 
\[
\fun^{p-1} \equiv 1 \mod{\p^2}.
\]
Thus, the AAC conjecture is related to the existence of number field
analogues of Wieferich primes. Therefore, 
in the second part of this paper, we investigate the infinitude of Wieferich primes in number fields.  Subject to Masser's $abc$-conjecture for number fields, we prove that for any fixed $\a\in \O_K\setminus\{0\}$ that is not a root of unity, there are infinitely many primes ideals $\p\subseteq \O_K$ for which
\[
\a^{N(\mathfrak{p})-1} \not\equiv 1 \mod{\mathfrak{p}^2}.
\]
Additionally, we show under the weaker assumption that there are finitely many \textit{base-$\a$ super-Wieferich primes}, and that there are infinitely many base-$\a$ non-Wieferich primes. In both cases, we obtain the lower bound
\[
\#\left\{\text{prime ideals } \p : N(\p)\leq x \text{ and } \a^{N(\p)-1}\not\equiv 1 \mod{\p^2}\right\} \gg_{K, \a, \eps} \frac{\log x}{\log\log x}
\]
 as $x\to \infty$. 
\vskip 10pt
\noindent{\bf Keywords.} Ankeny--Artin--Chowla conjecture, Mordell conjecture, Siegel's theorem,  Wieferich primes, $abc$-conjecture for number fields.\newline 
{\bf MSC 2020.} 11R04 (Primary); 	11N69, 11G05 (Secondary)} 


\section{Introduction}

Given a fixed non-zero integer $a \neq \pm 1$, a \textit{base-$a$ Wieferich prime} is a prime number $p$ (not dividing $a$) such that
\[
a^{p-1}\equiv 1 \mod{p^2}.
\]
When $a=2$, these numbers were classically studied by A. Wieferich in 1909 in connection with Fermat's Last Theorem \cite{Wieferich1909}.  Despite this classical connection, Wieferich primes still remain a mystery.  Part of the difficulty in understanding base-$a$ Wieferich primes is their sparsity.  Assuming that the quotient $\frac{a^{p-1}-1}{p}$ behaves like a random integer, the ``probability" it is divisible by $p$, is $1/p$. Therefore, the number of primes less than $x$ that are base-$a$ Wieferich primes is expected to be at most $O(\log\log x)$. In particular, for a fixed integral base, we would expect the density of base-$a$ Wieferich primes to have a natural density zero among the set of all primes.  We refer the interested reader to \cite{Crandall1997,Gras2016a, Gras2016b, Gras2018, Katz2015} for further discussions on probabilistic models concerning Wieferich primes and other algebraic quantities. \newline

A \textit{base-$a$ non-Wieferich prime} $p$ is a prime for which $a^{p-1}\not \equiv 1\mod{p^2}$. Given the above heuristic, we would expect that for any fixed base $a$, almost all primes are base-$a$ non-Wieferich primes. There is currently no unconditional result in this direction. The approach that has received the most attention to date is J. Silverman's use of the $abc$-conjecture to prove that there are infinitely many base-$a$ non-Wieferich primes in a quantitative form by obtaining a lower bound of $\log x$ for the number of primes less than $x$ that are base-$a$ non-Wieferich primes \cite{Silverman1988}. There have been several works that obtain the infinitude of non-Wieferich primes under weaker hypotheses than the $abc$-conjecture (see for example \cite{Granville1986, Murty2019, Schinzel1976}). \newline 

Using a variation of Silverman's argument, H. Graves and M. R. Murty \cite{Graves2013} proved an analogous conditional result with a weaker lower bound for base-$a$ Wieferich primes lying in certain arithmetic progressions. Recently, Y. Ding in \cite{Ding2019} has improved upon the result of Graves and Murty to obtain a lower bound of $\log x$.  \newline

Let $K/\Q$ be a finite algebraic number field and $\O_K$ denote its ring of integers. For $x\in \O_K$, denote by $N(x)=|N_{K/\Q}(x)|$ the absolute value of the norm of $x$. We will call an element $\a\in \O_K$  an \textit{admissible Wieferich base} if it is non-zero and not a root of unity. In analogy with the integer case, a non-zero prime ideal $\p\subseteq \O_K$ is a \textit{base-$\a$ Wieferich prime} if 
\[
\a^{N(\p)-1}\equiv 1 \mod{\p^2}.
\]
As with the integer case, we would expect the number of base-$\a$ Wieferich primes to be sparse. Indeed, by Lagrange's theorem, $N(\a^{N(\p)-1}-1)/N(\p)$ is an integer. If we assume that these quotients are randomly distributed among the residue classes $\mod{N(\p)}$, then the probability that it is the zero class is $1/N(\p)$. By Merten's theorem for number fields we deduce that 
\begin{align*}
    \# \{\text{prime ideals $\p\subset \O_K$} : N(\p) \leq x \text{ and } \a^{N(\p)-1} \equiv 1\mod{\p^2} \}\\ \approx \sum_{N(\p)\leq x} \frac{1}{N(\p)} =O( \log \log x)
\end{align*}
as $x\to \infty$. 
\subsection{Statement of results}

In contrast to the integer case which was motivated by investigations into Fermat's last theorem, the study of Wieferich primes in number fields has historically been driven by their connection to Euclidean algorithms \cite{Clark1992, Clark1995, Harper2004, Murty2018}. The present work was motivated by the following observation about the fundamental units of real quadratic fields (cf. Proposition 3.2 of \cite{Bouazzaoui2020}).  
\begin{thm}\label{thm: AACM thm}
     Let $d>2$ be a square-free positive integer, $K= \Q(\sqrt{d})$ be the real quadratic field of discriminant $\delta^2d$ with $\delta = 1$ if $d\equiv 1 \mod{4}$ and $\delta=2$ if $d=2,3\mod{4}$, and 
\[
\fun = \frac{\delta}{2} \left(t+ u\sqrt{d} \right)
\]
be the fundamental unit of $K$. Then $d\nmid u$ if and only if for some odd prime divisor $p$ of $d$, 
\[
\fun^{p-1} \not\equiv 1 \mod{\p^2}
\]
where $\p\cap \Z = p$ is the prime lying above $p$ in $\O_K$. In particular, $\p$ is a base-$\fun$ non-Wieferich prime. 
\end{thm}

When $d=p$ is an odd rational prime, N. Ankeny, E. Artin, and S. Chowla \cite{Ankeny1951, Ankeny1952} (when $p\equiv 1\mod{4}$) and independently L. J. Mordell \cite{mordell1961} (when $p\equiv 3\mod{4}$) conjectured that $p\nmid u$. Recently, A. Reinhart \cite{reinhart2024b, reinhart2024} has found primes that violate both of these conjectures. Counterexamples for square-free composite numbers are well-known \cite{Mollin1986}. For a discussion on the results of Ankeny, Artin, and Chowla, as well as interpretations of the known counterexamples, we refer the reader to \cite{Fellini2025}.\newline

If we denote the fundamental unit of $\Q(\sqrt{p})$ by $\fun_p$ and assume that the quotients $N(\fun^{p-1}-1)/p$ are randomly distributed $\mod{p}$, then as above we would expect that
\begin{align*}
    \#\{p\leq x : \p \text{ is ramified and a base-$\fun_p$ Wieferich prime in $\Q(\sqrt{p})$}\}\\ \approx \sum_{p\leq x} \frac{1}{p} = O(\log \log x)
\end{align*}
as $x\to \infty$. Theorem \ref{thm: AACM thm} then implies that this also estimates the number of counterexamples to the Ankeny--Artin--Chowla and Mordell conjectures. Moreover, this estimate agrees with the heuristic argument of L. Washington \cite{Washington}.
\newline
 
Given the expected similarities between the rational and number field settings it seems fruitful to formulate and prove infinitude results in the spirit of the results mentioned in the introduction. Our first result is a direct qualitative generalization of Silverman's result.  

\begin{thm}\label{thm: thm 1}
    Let $K$ be an algebraic number field and $\alpha\in \O_K$ be an admissible Wieferich base. Assuming Masser's $abc$-conjecture for number fields, there are infinitely many prime ideals $\p \subset \O_K$ such that 
    \[
    \a^{N(\p)-1} \not\equiv 1 \mod{\p^2}. 
    \]
\end{thm}
In the proof of Theorem \ref{thm: thm 1}, we will see that it is essential that $N(\a^n-1)\to \infty$ as $n\to \infty$. In the classical setting, verifying this is a routine calculus exercise. For a general number field, this norm condition is comparatively subtle. In Section \ref{sec: an application of Siegel's theorem}, we will prove, using Siegel's theorem on integral points, that any admissible Wieferich base satisfies this norm condition.   \newline 

There are several known cases of Theorem \ref{thm: thm 1} in the literature \cite{Graves2025, Ichimura1998, Srinivas2018}.  In the case of \cite{Ichimura1998}, it is clear that the author had such a norm condition in mind, but the given explanation overlooks the delicate nature of the quantities involved. Indeed, their inequality (2) is far from obvious and seems to presuppose that there are only finitely many solutions to the equation $u+v=1$ in $\O_K^\times$, a fact that is a straightforward consequence of Siegel's theorem. In \cite{Srinivas2018}, the authors impose Galois theoretic conditions  (as well as a class number one assumption) on the bases considered to bypass the norm condition. These restrictions allow only specific units to be considered. \newline  

While preparing this manuscript, we were made aware of a preprint by H. Graves and B. Weiss \cite{Graves2025}. In this paper, the norm condition is bypassed using Galois theoretic hypotheses and additional bounds for the norm under these conditions. Similar to \cite{Srinivas2018}, these hypotheses restrict the bases that can be considered. Notably, the Theorem 1 of \cite{Graves2025} contains a gap when the bases considered are units. Indeed, their Lemma 6 gives, in essence, a trivial lower bound which doesn't force any growth for the norms of the ideals considered. The growth of these norms is essential for proving infinitude. \newline

We will call a prime ideal $\p$ a \textit{base-$\a$ super-Wieferich prime} if 
\[
\a^{N(\p)-1} \equiv 1 \mod{\p^3}.
\]
Following the heuristic in the beginning of the introduction, the number of base-$\a$
super-Wieferich primes is expected to be bounded. Following the work of M. R. Murty and F. S\'eguin \cite{Murty2019}, we obtain the following.  
\begin{thm}\label{thm: 3}
    Let $K$ be an algebraic number field and $\alpha\in \O_K$ be an admissible Wieferich base. If there are finitely many base-$\a$ super-Wieferich primes, then there are infinitely many base-$\a$ non-Wieferich primes. 
\end{thm}
The interested reader should consult J. Voloch's article \cite{Voloch2000} for a complementary result. \newline

In both settings of Theorem \ref{thm: thm 1} and Theorem \ref{thm: 3}, we can obtain the following lower bound. 
\begin{thm} \label{thm: thm 2}
    Let $K$ be an algebraic number field and $\alpha\in \O_K$ be an admissible Wieferich base. If the $abc$-conjecture is true or there are finitely many base-$\a$ super-Wieferich primes then  
    \[
   \# \{\p \in O_K: N(\p) \leq x \text{ and } \a^{N(\p)-1} \not\equiv 1\mod{\p^2}\} \gg \frac{\log x}{\log\log x}
    \] 
    as $x\to \infty$. Here, the implied constant can depend on the field $K$, the base $\a$, and in the case of the $abc$-conjecture, $\eps$. 
\end{thm}

\section{Connection with the Ankeny--Artin--Chowla and Mordell conjectures}\label{sec: AAC & Wieferich}
In this section, $d>2$ will denote a square-free integer, $\Q(\sqrt{d})$ is the real quadratic field of discriminant $\delta^2d$ with $\delta =1 $ if $d\equiv 1 \mod{4}$ and $\delta=2$ if $d\equiv 2, 3\mod{4}$, $p$ will denote any odd prime divisor of $d$, and $\p$ will denote the prime above $p$ in $\O_K$.  
Thus, $\p^2=p\O_K$.  We will write
\[
\fun = \frac{\delta}{2} \left(t+ u\sqrt{d}\right)
\]
for the fundamental unit of $\Q(\sqrt{d})$. Here $t$ and $u$ are integer solutions to the equation $t^2-du^2 =\pm 4$ chosen so that $t+u\sqrt{d}>1$ is minimal among all solutions when considered as a real number. 

\begin{proof}[Proof of Theorem \ref{thm: AACM thm}]
    Suppose $p$ is an odd prime divisor of $d$. Write $\fun^n = t_n + \sqrt{d}u_n$. It will suffice to show that $u_n\equiv 0 \mod{p}$ if and only if $p\mid u$ or $n\equiv 0\mod{p}$. Indeed, for any positive integer $n$, we have that 
    \[ t_n + \sqrt{d}u_n =
    \fun^n = \frac{\delta^n}{2^n}(t+u\sqrt{d})^n = \left(\frac{\delta}{2}\right)^n \sum_{j=0}^n { n\choose j } t^{n-j}(u\sqrt{d})^{j}.
    \]
    Splitting the sum into even and odd indices, we obtain
    \[
    \fun^n = \left(\frac{\delta}{2}\right)^n\left( \sum_{j=0}^{\lfloor \frac{n}{2} \rfloor} {n \choose 2j} t^{n-2j}d^ju^{2j} + \sqrt{d}\sum_{j=0}^{\lfloor \frac{n}{2} \rfloor} {n \choose 2j+1} t^{n-2j-1}d^j u^{2j+1}\right).
    \]
    Hence, 
    \[
    t_n =  \left(\frac{\delta}{2}\right)^n\left( \sum_{j=0}^{\lfloor \frac{n}{2} \rfloor} {n \choose 2j} t^{n-2j}d^ju^{2j} \right) \equiv \left(\frac{\delta t}{2}\right)^n \mod{p}
    \]
    and
    \[
    u_n = \left(\frac{\delta}{2}\right)^n\left(\sum_{j=0}^{\lfloor \frac{n}{2} \rfloor} {n \choose 2j+1} t^{n-2j+1}d^j u^{2j+1}\right) \equiv n\left(\frac{\delta t}{2}\right)^n \frac{u}{t} \mod{p}. 
    \]
    We see at once, that as $\gcd(t,p)=1$, $u_n\equiv 0\mod{p}$ if and only if $p \mid u$ or $n\equiv 0 \mod{p}$. In particular, for $n=p-1$ we see that $t_{p-1}\equiv 1 \mod{p}$ and $u_{p-1}\equiv-\frac{u}{t} \mod{p}$. Hence $u_{p-1}\equiv 0 \mod{p}$ if and only if $p\mid u$. Recalling that $\p^2=(p)$, we have
    \[
    \fun^{p-1}\equiv 1 \mod{\p^2}
    \]
    if and only if $p\mid u$. 
\end{proof}

\section{Heights and the $abc$-conjecture in number fields}\label{sec: the abc-conjecture}
Let $K$ be an algebraic number field. We will say that two absolute values $v_1$ and $v_2$ on $K$ are equivalent if and only if they generate the same topology on $K$. A place $v$ of $K$ is an equivalence class of absolute values under this relation. A \textit{finite place}, or non-Archimedean place, arises from the valuation associated to each prime ideal of $\O_K$. An \textit{infinite place}, or Archimedean place, arises from embeddings of $K$ into the complex numbers. If $\sigma$ is a real embedding, we call it a real infinite place. Each complex conjugate pair of embeddings $\{\sigma, \overline{\sigma}\}$ determine a complex infinite place. If $K/\Q$ has $r_1$ real embeddings and $2r_2$ complex embeddings, there are $r_1+r_2$ infinite places of $K$. We denote the set of all  finite places of $K$ by $M_K^0$ and the set of infinite places by $M_K^\infty$ so that the set of all places of $K$, $M_K$, is $M_K= M_K^0 \cup M_K^\infty$. \newline

We will denote by $N(a):= |N_{K/\Q}(a)|$ the absolute value of the norm of $a\in K$. For each (finite and infinite) place $v$ of $K$, we normalize the corresponding valuations so that for every $x\in K^\times$ the product formula 
\[
\prod_{v\in M_K} ||x||_v =1
\]
holds. Precisely, for each finite place $v$ of $K$ and corresponding prime ideal $\p_v$, we set $||x||_v = N(\p_v)^{-\operatorname{ord}_{\p_v}(x)}$ where $\operatorname{ord}_{\p_v}(x)$ is the exponent of the ideal $\p_v$ in the prime ideal factorization of principal ideal generated by $x$. For the infinite places,  we normalize so that $||x||_v = |x|_v^{d_v}$ where $d_v=1$ or $2$ according to whether $v$ is real or complex.\newline 

Unlike the $abc$-conjecture for $\Q$, there have been several proposed generalizations for number fields, notably, by P. Vojta \cite{Vojta1987}, N. Elkies \cite{Elkies1991}, G. Frey \cite{Frey1997}, and D. Masser \cite{Masser2002}. We will use the version proposed by Masser as it exhibits a nice functorial property under field extensions. We formulate Masser's $abc$-conjecture for number fields now. \newline 

Suppose that $K/\Q$ is an algebraic number field and that $a,b,c\in K^\times$. The \textit{height} of the tuple $(a,b,c)$ is the product
\[
H_K(a,b,c) = \prod_{v\in M_K} \max(||a||_v, ||b||_v, ||c||_v).
\]

The \textit{ramified support} of the tuple $(a,b,c)\in (K^{\times})^3$ is defined as
\[
\supp_K(a,b,c) = \prod_{\p \in \O_K} N(\p)^{e(\p|p)}
\]
where the product is over those prime ideals $\p$ for which $||a||_\p, ||b||_\p$, and $||c||_\p$ are all distinct and $e(\p|p) = \ord_\p(p)$ is the ramification index of $\p$ over $p$. The empty product will be interpreted simply as being $1$. Similarly, the \textit{ramified support} of an ideal $I\subseteq \O_K$ is
\[
\supp_K(I) := \prod_{\p\mid I}N(\p)^{e(\p|p)}.
\]

If $a,b,c\in \Q$, then it is well known that $H_K(a,b,c) = H_\Q(a,b,c)^{[K:\Q]}$. The benefit of the ramified support is that it shares this property. Indeed, it is a straightforward calculation (which can be found in \cite{Masser2002}) that  
if $a,b,c\in \Q$ then $\supp_K(a,b,c) = \supp_\Q(a,b,c)^{[K:\Q]}$. For both the height and support, similar identities hold for any finite extension of number fields $L/K$ \cite{Masser2002}. This is the functorial property we alluded to above. Masser's $abc$-conjecture for number fields asserts: 
\begin{conj}[$abc$-conjecture for number fields]\label{conj: abc-conjecture}
    Let $K/\Q$ be a number field with discriminant $\Delta_K$ and $\eps>0$. Suppose $a+b=c$ for $a,b,c\in K^\times$. Then there is a constant $C_{\eps}>0$ such that 
    \[
    H_K(a,b,c) \leq C_{\eps} (\Delta_K \supp_K(a,b,c))^{1+\eps}. 
    \]
\end{conj}

We note that both Elkies' and Frey's version of the $abc$-conjecture use what we would call the \textit{divisor support} which is defined by a similar product as the ramified support \cite{Elkies1991, Murty2012}. For our results, Elkies' version  of the $abc$-conjecture gives a slightly simplified proof as the divisor support can be bounded in a straightforward way. 
\newline

We will require the following lower bound for the height function. 
\begin{lem}\label{lem: lower bound for height}
    Suppose $a,b, c \in \O_K^\times$, that $a,b,$ and $c$ are pairwise coprime, and  $a+b=c$. Then 
    \[
     \max (N(a), N(b), N(c)) \leq H_K(a,b,c).
    \]
\end{lem}
\begin{proof}
    Since $a,b,c$ are pairwise coprime, for each finite place $v$, at least one of $||a||_v, ||b||_v, ||c||_v$ is equal to 1. Therefore, 
    \[
    H_K(a,b,c) \geq \prod_{v \in M_K^\infty}  \max(||a||_v, ||b||_v, ||c||_v).
    \]
    By the product formula, the absolute norm of a non-zero element $a\in \O_K$ is given by
    \[
    N(a) = \prod_{v \in M_K^\infty} ||a||_v.
    \]
    Therefore, 
    \[
    H_K(a,b,c) \geq N(a)\cdot \prod_{v \in M_K^\infty}  \max(1, ||b/a||_v, ||c/a||_v) \geq N(a). 
    \]
    Interchanging the role of $a,b$ and $c$, we deduce the result. 
\end{proof}
Just as we defined the height of a tuple $(a,b,c)\in (K^\times)^3$, we can define the the \textit{absolute multiplicative height} of $\a\in K^\times$ by
\[
h(\a) = \prod_{v\in M_K} \max(1, ||\a||_v)^{1/[K:\Q]}. 
\]
This height function and its logarithm have many nice properties and interesting applications. We refer the interested reader to the books \cite{WaldschmidtBook} and \cite{EvertseBook} for more on heights and their uses. In particular, we have the useful upper bound (which should be compared with the hypothesis of Lemma 4 of \cite{Graves2025}). 
\begin{lem}\label{lem: upperbound for norm}
    Suppose $\a\in K^\times$ and that $\a$ is not a root of unity. Then 
    \[
    N(\a^n-1)  \leq 2h(\a)^n.
    \]
\end{lem}
\begin{proof}
    This follows immediately from inequality (1.9.2) and Lemma 1.9.2 of \cite{EvertseBook} upon taking $S=~M_K^\infty$. 
\end{proof}

\section{An application of Siegel's Theorem}\label{sec: an application of Siegel's theorem}
For a number field $K$ with ring of integers $\O_K$, we will call a set $S$ of elements in $\O_K$  \textit{multiplicatively closed} if: $1$ is in $S$ and if $x,y\in S$, then $xy\in S$.  We will make the further assumption that $0$ is not in $S$. Then we define the \textit{localization} $S^{-1}\O_K$ to be the set
\[
S^{-1}\O_K := \left\{\frac{a}{s}: a\in \O_K, s\in S\right\} \subset K
\]
under the equivalence $\frac{a}{s_1}\sim\frac{b}{s_2}$ if and only if
\[
as_2 - bs_1 = 0.
\]
We will write $\O_{K,S} := S^{-1}\O_K$. The ring $\O_K$ has a canonical inclusion in $\O_{K,S}$ and we call $\O_{K,S}$ the ring of $S$-integers.  We can alternatively define the ring of $S$-integers using a set of places. Indeed, if $S\subset M_K$ is a finite set of places of $K$, then 
\[
\O_{K,S} = \{a\in K : ||a||_v \leq 1 \text{ for all } v\notin S \}. 
\]

\begin{prop}[\protect{\cite[Proposition 16]{LangANT}}]\label{prop: O_{K,S} is a dedekind domain}
    Let $S$ be a multiplicative set of $\O_K$, then $\O_{K,S}$ is a Dedekind domain. The map
    \[
    \mathfrak{a} \mapsto S^{-1}\mathfrak{a}
    \]
    is a surjective group homomorphism from the group of fractional ideals of $\O_K$ onto the group of fractional ideals of $\O_{K,S}$. Moreover, the kernel of this map is precisely the fractional ideals $\mathfrak{a} \subset \O_K$ such that $\mathfrak{a} \cap S\neq 0$.
\end{prop}
The prime ideals of $\O_S$ are ideals of the form $S^{-1}\p$ where $\p$ is a prime ideal of $\O_K$ which do not intersect $S$. If $S^{-1}\p$ is a prime ideal in $\O_S$, then its preimage under the map from the proposition is the ideal $\p$. 

A celebrated theorem of Hecke (see for example page 149 of \cite{murty-esmonde})
that every ideal class contains infinitely many prime ideals.  Since the class number is finite, we may choose a prime ideal from each ideal class to comprise a finite set $S$.  This observation enables us to state the following useful theorem.

\begin{prop}[\protect{\cite[Proposition 17]{LangANT}}]\label{prop: O_{K,S} has class number 1}
    There is a finite multiplicative set $S$ of $\O_K$ such that $\O_{K,S}$ has class number 1. 
\end{prop}

We recall Siegel's theorem in the form of Corollary IX.3.2.1 from \cite{Silverman2009}. 

\begin{thm}[Siegel (1929)]
    Let $E/K$ be an elliptic curve with Weierstrass coordinates $x$ and $y$, $S\subset M_K$ a finite set of places including all of the infinite places of $K$, and $\O_{K,S}$ the ring of $S$-integers of $K$. Then 
    \[
    \{P\in E(K) : x(P)\in \O_{K,S}\}
    \]
    is a finite set. That is, $E$ has only finitely many points with $x$-coordinates in $\O_{K,S}$. 
\end{thm}
One can find more about Siegel's theorem, its applications, and generalizations in \cite{Hindry2000}.

\begin{prop} \label{prop: norm to infty}
    Let $K/\Q$ be a number field and $\a\in \O_K$ be an admissible Wieferich base. For any $M>0$, there are only finitely many positive integers $n$ such that 
    \[
    N(\a^n-1) \leq M.  
    \]
\end{prop}
\begin{proof}
    Towards a contradiction, suppose there exists some $M>0$ such that there are infinitely many $n$ for which
    \[
    N(\a^n-1) \leq M.
    \]
    
    From Proposition \ref{prop: O_{K,S} has class number 1}, there is a multiplicative set $S$ such that $\O_{K,S}$ is a PID. In particular, we can take $S$ to be the multiplicative set of all finite places associated to the prime ideal factors of each integral ideal representative of the different ideal classes of $\operatorname{Cl}(K)$ along with all of the infinite places of $K$. By Proposition \ref{prop: O_{K,S} is a dedekind domain},  $\O_{K,S}$ is a Dedekind domain.  Thus there are only finitely many prime ideals of bounded norm and hence finitely many prime elements with norm less than $M$. We enumerate these prime elements as $\pi_1, \ldots, \pi_r$.\newline 
    
    For each $n$ for which
    \[
    N(\a^n-1) \leq M, 
    \]
    we write 
    \[
    \a^n -1 = \eta \pi_1^{c_1}\cdots \pi_r^{c_r}
    \]
    where $\eta\in\O_{K,S}^{\times}$, $\pi_i\in \O_{K,S}$ are prime elements, and $c_i$ are non-negative integers. By Dirichlet's unit theorem for $S$-integers (Corollary 1.8.2 of \cite{Evertse1984}), we can write 
    \[
    \eta = \zeta \fun_1^{b_1}\cdots \fun_{s-1}^{b_{s-1}} 
    \]
    where $\zeta$ is a root of unity in $K$, $\{\fun_i\}_{i=1}^{s-1}$ are a fundamental system of units of $\O_{K,S}^\times$, and $b_i$ are non-negative integers. Then we write, using the division algorithm, 
    \[
    b_i = 3x_i + \beta_i \text{ and  } c_i = 3y_i +\gamma_i 
    \]
    where $\beta_i, \gamma_i \in \{0,1,2\}$. Similarly, for each value of $n$ we will write 
    \[
    n = 2z_n + \theta_n, 
    \]
    where $\theta_n\in \{0,1\}$. Since $N(\a^n-1)\leq M$, there are only finitely many choices for the $c_i$. Then we consider the family of curves given by 
    \[
    A y^2 = BX^3 +1
    \]
    where 
    \[
    A = \a^{\theta_n}
    \]
    and 
    \[
    B= \zeta \fun_1^{\beta_1}\cdots \fun_{s-1}^{\beta_{s-1}}\pi_1^{\gamma_1}\cdots \pi_r^{\gamma_r}.
    \]
    There are at most $\# W_K \cdot 2\cdot r^3\cdot (s-1)^3$ curves in this family, where $s$ is the unit rank of $\O_{K, S}$, and $W_K$ is the set of roots of unity in $K$. By Siegel's theorem, each of these curves has only finitely many $S$-integer points. However, by construction, each $n$ for which $N(\a^n-1)\leq M$ produces a distinct $\O_{K,S}$ point on one of the finitely many elliptic curves above.  This contradicts the finiteness from Siegel's theorem and hence there must be only finitely many $n$ for which $N(\a^n-1)\leq M$.     
\end{proof}
As an immediate consequence,  we obtain the following norm condition for admissible Wieferich bases. 
\begin{cor}\label{cor: norm to infty}
    If $\a\in \O_K$ is an admissible Wieferich base, then $N(\a^n-1)\to \infty$ as $n\to \infty$. 
\end{cor}

\section{The factorization of cyclotomic ideals}
Fix an algebraic number field $K/\Q$ and a non-zero non-root of unity $\a\in \O_K$. Let $\Phi_n(x)$ denote the $n$-th cyclotomic polynomial. Then it is well known that for any algebraic number $\a$, 
\begin{align}\label{eq: factorization}
    \a^n-1 = \prod_{d\mid n} \Phi_n(\a).
\end{align}
In this section we study the prime ideal factorization of $A_n(\a)=(\a^n-1)$. We will call an ideal a \textit{cyclotomic ideal} if it is a principal ideal generated by the value of a cyclotomic polynomial. We denote the $n$-th cyclotomic ideal generated by $\Phi_n(\a)$ by $C_{n}(\a)$. To ease notation we will simply use $A_n$ and $C_n$ when the base being used is clear. We begin with the following elementary lemma. 

\begin{lem}\label{lem: product of principal ideals}
    Suppose $R$ is a commutative ring with multiplicative identity $1$. If $a,b\in R$, then $(a)(b) = (ab)$. That is, the product of two principal ideals is the ideal generated by the product of the two generators. 
\end{lem}
\begin{proof}
   Since $1\in R$, we have that $ab\in (a)(b)$ and hence $(ab) \subseteq (a)(b)$. On the other hand, 
    \[
    (a)(b) =\left\{ab\sum x_iy_i : x_i,y_i\in R\right\}
    \]
    from which we deduce that $(a)(b) \subseteq (ab)$ since we may take $x_i=y_i=1$. 
\end{proof}
Applying the lemma to the factorization in equation (\ref{eq: factorization}) we immediately obtain:
\begin{cor}\label{cor: cyclo. ideal factorization}
    If $\a\in \O_K$, then
    \[
    A_n = \prod_{d\mid n} C_d.
    \]
\end{cor}

As in Section \ref{sec: the abc-conjecture}, for a non-zero prime ideal $\p\subset \O_K$ and any ideal $I$ of $\O_K$, we define $\ord_\p(I)$ to be the largest power of $\p$ occurring in the prime ideal factorization of $I$. If $I$ is the zero ideal, we set $\ord_\p(I)=\infty$. Similarly, for any rational prime $p$ and integer $n$, we define $\ord_p(n)$ to be the highest power of $p$ dividing the integer $n$. We write $\p|p$ to denote that the prime $\p$ lies above the rational prime $p$ and $e(\p|p) = \ord_\p(p)$ for the ramification index of $\p|p$.
\newline

Suppose $\p$ is a non-zero prime ideal and that $\a\notin \p$. The order of $\a \mod{\p}$  is the smallest positive integer $f_\a(\p)$ for which 
\[
\a^{f_\a(\p)} \equiv 1 \mod{\p}. 
\]
By Lagrange's theorem, $f_\a(\p)$ is bounded by $N(\p)-1$. Furthermore, let $\delta_\a(\p)$ denote the exact power of $\p$ dividing $A_{f_\a(\p)}$, i.e., $\delta_\a(\p)= \ord_\p(A_{f_\a(\p)})$. 

\begin{lem} \label{lem: exact order}
    Suppose $\a\not \in \p$. Then $\ord_\p(C_{f_\a(\p)}) = \delta_{\a}(\p)$.
\end{lem}
\begin{proof}
    By definition, $\p^{\delta_{\a}(\p)}$ exactly divides $A_{f_{\a}(\p)}$. By Corollary \ref{cor: cyclo. ideal factorization}, $\p^{\delta_\a(\p)}$ exactly divides 
    \[
    \prod_{d\mid f_\a(\p)} C_d
    \]
    and hence $\p$ must divide $C_d$ for at least one divisor $d$ of $f_\a(\p)$. If $\p\mid C_d$ for any proper divisor $d$ of $f_\p$, then $\p$ divides
    \[
     C_d \prod_{\substack{e\mid d\\e\neq d}} C_e = A_d.
    \]
    Since $d<f_{\a}(\p)$, this contradicts the minimality of $f_\a(\p)$. Hence, $\ord_\p(C_{f_\a(\p)}) = \delta_\a(\p)$. 
\end{proof}

This determines the power of $\p$ that occurs in the factorization of $C_n$ when $n=f_\a(\p)$. We follow an argument of E. Artin  \cite{Artin1955} to classify the powers of $\p$ that can divide the cyclotomic ideals when $n\neq f_\a(\p)$. 
\begin{lem}\label{lem: p-adic valuation of ideal}
   Let $K/\Q$ be an algebraic number field, $\a\in \O_K$ be non-zero and not a root of unity, and $p$ an odd unramified rational prime.  If $\p|p$ is a prime ideal, then
    \[
    \ord_\p (C_n) = \begin{cases}
    \delta_\a(\p) & \text{ if $n=f_\a(\p)$,}\\
    1 & \text{  if $n=p^if_\a(\p) $ for any $i\geq 1$, and}\\
    0 & \text{  otherwise.}
    \end{cases}
    \]
\end{lem}
\begin{proof}
The first case is handled by the previous lemma. We will dispense of the third case next. Fix a rational prime $p$. Suppose $n=p^if_{\a}(\p)m$ for some $i\geq 0$ and integer $m>1$ coprime to $p$. Set $r=p^if_{\a}(\p)$. Then, 
    \[
    \frac{\a^n-1}{\a^r-1} =  \sum_{k=0}^{m-1} \a^{r k}.
    \]
    Since $\a^{f_\a(\p)} \equiv 1\mod{\p}$ and $\gcd(p,m)=1$ we deduce that 
    \[
    \frac{\a^n-1}{\a^r-1} \equiv m \not \equiv 0 \mod{\p}.
    \]
    Since $\Phi_n(\a)$ divides $\frac{\a^n-1}{\a^r-1}$, we deduce that $\ord_\p(C_n)=0$.  \newline

   In the second case, we set $n=p^if_{\a}(\p)$ and $r=p^{i-1}f_{\a}(\p)$ for some $i\geq 1$. Then, 
    \[
        \Phi_n(\a) = \frac{\a^{n}-1}{\a^r-1} = \sum_{k=1}^{p}{p\choose k}(\a^r-1)^{k-1}.
    \]
    For $2\le k\leq p$, we have that
    \[
    \ord_\p\left( {p\choose k}(\a^r-1)^{k-1} \right) \geq e(\p|p) + (k-1)\delta_{\a}(\p) \geq 2
    \]
    since $e(\p|p)=1$ and $(k-1)\delta_\a(\p) \geq 1$. Therefore, 
     \[
        \Phi_n(\a) \equiv p\not\equiv 0 \mod{\p^2}
    \]
    as $p$ is unramified.

\end{proof}
We remark that the first and third parts of the lemma remain unchanged if we consider odd ramified primes or primes above $2$.  In the odd ramified case, it is evident from the above calculation that $\ord_\p(C_{p^if_\a(\p)} )\geq 2$. For $\p|2$, we have that $\ord_\p(C_{2^if_\a(\p)})>0$. \newline 

The previous lemma allows us to quickly determine the following:
\begin{lem}\label{lem: gcd of ideals}
    Suppose $\gcd(m,n)=1$. Then $\gcd( A_m, A_n) = (\a-1)\O_K$. 
\end{lem}
\begin{proof}
If $m=n=1$, then the result is trivially true.
    Suppose now $\gcd(m,n)=1$ and that at least one of $m$ or $n$ is greater than $1$. Without loss of generality assume $m>1$. Using the factorizations 
    \[
    A_m = \prod_{d\mid m} C_m \text{ and } A_n = \prod_{e\mid n} C_n. 
    \]
    it is clear that $C_1 = A_1$ divides both $A_m$ and $A_n$. Suppose $d$ and $e$ are non-trivial divisors of $m$ and $n$ respectively. If $\p\mid C_d$ the previous lemma implies that $d=f_\a(\p)p^i$ for some $i\geq 0$. Since $d$ and $e$ are coprime we deduce that $\ord_\p(C_e)=0$. Therefore, the cyclotomic ideals are coprime for coprime indices and hence, 
    \[
    \gcd(A_m, A_n) = A_1
    \]
    as desired. 
\end{proof}

We will call a non-zero ideal $I$ square-free if $\ord_\p(I) \leq 1$ for all prime ideals $\p$. We will say $I$ is powerful if $\ord_\p(I)\geq 2$ for all prime ideal divisors $\p$ of $I$. 
\begin{lem}\label{lem: non-wieferich criteria}
    Suppose $A_n = U_nV_n$ where $U_n$ is square-free and $V_n$ is powerful. If $\p\mid U_n$, then 
    \[
    \a^{N(\p)-1}\not\equiv 1 \mod{\p^2}. 
    \]
\end{lem}
\begin{proof}
    Let $f_\a(\p)$ denote the order of $\a\mod{\p}$ and write 
    \[
    A_n = \prod_{d\mid n} C_n. 
    \]
    Since $\p\mid U_n$ we have that 
    \[
   \ord_\p(A_n) = \sum_{d\mid n} \ord_\p(C_d)=1.
    \]
   Lemma \ref{lem: p-adic valuation of ideal} implies that $n=mf_\a(\p)$ for some positive integer $m$ coprime to the rational prime $p$ below $\p$. Thus, $\ord_{\p}(C_{f_\a(\p)})=1=\delta_\p$ by Lemma \ref{lem: exact order}. Moreover, 
   \[
   A_{N(\p)-1} = C_{f_{\p}(\a)}\prod_{\substack{d\mid (N(\p)-1)\\d\neq f_\a(\p)} } C_{d}.
   \]
Since $N(\p)-1$ is coprime to the rational prime $p$, Lemma \ref{lem: p-adic valuation of ideal} implies that $\ord_\p(C_d)=0$ for each $C_d$ in the product. Hence, $\ord_{\p}(A_{N(\p)-1}) = \ord_\p(C_{f_\a(\p)})=1$.   
\end{proof}

\section{Proof of Theorem \ref{thm: thm 1}}
We will require a couple of elementary lemmas before proving Theorem \ref{thm: thm 1}. 
\begin{lem}\label{lem: divisors of square-free}
    Suppose $\{\b_n\}_{n\geq 1}$ is a sequence of square-free ideals of $\O_K$. If $N(\b_n) \to \infty$ as $n\to \infty$, then 
    \[
    \{\p \in \O_K: \p \text{ is a prime ideal and } \p\mid \b_n \text{ for some $n$} \}
    \]
    is unbounded. 
\end{lem}
\begin{proof}
    Towards a contradiction, suppose there are only finitely many primes in the above set, say $\p_1, \ldots, \p_r$. Then for any $n\geq 1$, 
    \[
    N(\b_n) \leq \prod_{i=1}^r N(\p_i), 
    \]
    since $\b_n$ is square-free. The product on the right is bounded and independent of $n$. This implies that $\lim_{n\to \infty} N(\b_n) < B <\infty$, a contradiction. 
\end{proof}

\begin{lem}\label{lem: bounds for support}
    Let $K/\Q$ be an algebraic number field and $\a\in O_K$ be an admissible Wieferich base. As above, let $(\a^n-1) = A_n = U_nV_n$. Then
    \[
    \supp_K(U_n) \leq N(U_n)^{[K:\Q]}
    \]
    and 
    \[
    \supp_K(V_n) \ll_K \sqrt{N(V_n)}. 
    \]
\end{lem}
\begin{proof}
    The first inequality is trivial upon noting that the ramification index of any prime ideal $\p$ is bounded by the degree of the extension $[K:\Q]$ and that $\ord_\p(U_n) = 1$ for all prime ideal divisors $\p$ of $U_n$. By definition we have that $\ord_\p(V_n)\geq 2$ for all $\p\mid V_n$. If $\p\mid V_n$ and $e(\p|p) =1$ we have that 
    \[
    \prod_{\substack{\p \mid V_n\\ e(\p|p)=1}} N(\p)^{2e(\p|p)} \leq   \prod_{\substack{\p \mid V_n\\ e(\p|p)=1}} N(\p)^{\ord_\p(V_n)}.
    \]
    If $\p \mid V_n$ and $e(\p|p)\geq 2$ we have that
    \[
   \prod_{\substack{\p \mid V_n\\ e(\p|p)\geq 2}} N(\p)^{2e(\p|p)} \leq \prod_{\substack{\p\mid V_n\\ e(\p|p) \geq 2}} N(\p)^{2[K:\Q] + \ord_\p(V_n)} \leq A^2 \prod_{\substack{\p\mid V_n\\ e(\p|p) \geq 2}} N(\p)^{\ord_p(V_n)}. 
    \]
    where 
    \[
    A = \prod_{\substack{\p\subset \O_K\\ e(\p|p)\geq 2}} N(\p)^{[K:\Q]}.
    \]
    We note that $A<\infty$ as there are only finitely many ramified primes in $\O_K$. In all, we deduce that 
    \[
    \supp_K(V_n)^2 =\prod_{\p \mid V_n} N(\p)^{2e(\p|p)} \leq A^2 \prod_{\substack{\p \mid V_n\\ e(\p|p)=1}} N(\p)^{\ord_\p(V_n)}\prod_{\substack{\p\mid V_n\\ e(\p|p) \geq 2}} N(\p)^{\ord_p(V_n)} = A^2 N(V_n). 
    \]
    Hence, $\supp_K(V_n) \ll_K \sqrt{N(V_n)}$.
\end{proof}

\begin{proof}[Proof of Theorem \ref{thm: thm 1}]
    Let $\a$ be an admissible Wieferich base for the field $K$. As in Lemma \ref{lem: non-wieferich criteria} we will write $A_n = (\a^n-1)=U_nV_n$ for $U_n$ square-free and $V_n$ powerful. Towards a contradiction, suppose that $N(U_n)$ is bounded as $n\to \infty$. We will show that this implies that $n$ must be bounded.     \newline
    
    Fix $\eps\in (0,1)$. Applying the $abc$-conjecture to the (tautological) equation
    \[
    1 + (\a^n-1) = \a^n
    \]
    we obtain the inequality
    \[
    H(1, \a^n, \a^n-1) \ll_{K, \eps, \a} \left(\supp(1, \a^n, \a^n-1) \right)^{1+\eps}. 
    \]
    It is clear that $1, \a^n-1, $ and $\a^n$ are pairwise coprime. Therefore, 
    \[
    \supp(1, \a^n, \a^n-1) = \supp(\a)\supp(\a^n-1). 
    \]
    By Lemma \ref{lem: bounds for support}, we deduce that 
    \[
    \supp(\a^n-1) = \prod_{\p\mid A_n}N(\p) \ll_K N(U_n)^{[K:\Q]}N(V_n)^{1/2}. 
    \]
    Inputting these bounds into the $abc$-conjecture and using Lemma \ref{lem: lower bound for height} we find 
    \[
    N(V_n)^{\frac{1-\eps}{2}} \ll_{K, \a, \eps} \supp(\a)^{1+\eps} N(U_n)^{(1+\eps)[K:\Q]-1}.
    \]
    Since $\a$ is fixed, $\supp(\a)$ is a fixed and by assumption, $\lim_{n\to \infty}N(U_n)<B$ for some finite constant $B$. Therefore, $N(V_n)$ must also be bounded as $n\to \infty$. This contradicts Corollary \ref{cor: norm to infty}. Hence, $\{U_n\}_{n\geq 1}$ is a sequence of square-free ideals with norm going to infinity. By Lemma \ref{lem: non-wieferich criteria} we conclude that there must be infinitely many base-$\a$ non-Wieferich primes. 
\end{proof}

\section{Proof of Theorem \ref{thm: 3}}
We will call a prime ideal $\p$ a base-$\a$ \textit{super-Wieferich prime} if 
\[
\a^{N(\p)-1} \equiv 1 \mod{\p^3}.
\]
\begin{thm}
    Suppose $\a$ is an admissible Wieferich base. If there are finitely many base-$\a$ super-Wieferich primes, then there are infinitely many base-$\a$ non-Wieferich primes.
\end{thm}
\begin{proof}
    Suppose there are finitely many base-$\a$ super-Wieferich primes. We will consider a set $T$ of odd rational primes whose members satisfy all of the the following conditions:
    \begin{enumerate}[label=(\roman*)]
        \item $q\in T$ if any prime $\q$ above $q$ is not a base-$\a$ super-Wieferich prime. 
        \item $q\in T$ if for any base-$\a$ super-Wieferich prime $\p$,  $f_\a(\p) \neq q$.
        \item $q\in T$ if no prime $\q$ above $q$ divides $A_1 = (\a-1)\O_K$.
       
    \end{enumerate}
    We note that as there are only finitely many base-$\a$ super-Wieferich primes and finitely many prime ideals dividing $A_1$, conditions (i)--(iii) remove only finitely many rational primes from consideration. Therefore, $T$ contains infinitely many rational primes. \newline 

   Condition (iii) ensures no prime $\q|q$ divides $C_q$. Indeed, suppose that $\q\mid C_q$. Then, 
   \[
   \a^q \equiv 1 \mod{\q}. 
   \]
    Since $f_\q$ is the order of $\a \mod{\q}$, we must have that $f_\q \mid q$. As $q$ is a rational prime, this implies that $f_\a(\q)=q$ or $f_\a(\q)=1$. By Lagrange's theorem, the first case cannot happen as $N(\q)-1$ and $q$ are coprime. In the second case, we deduce that $\a\equiv 1 \mod{\q}$.  \newline

  Suppose $\p|p$ divides $C_q$ for some $q\in T$. Then Lemma \ref{lem: p-adic valuation of ideal} implies that $q = p^if_\a(\p) $ for some $i\geq 0$. We note that we cannot have that $f_\a(\p)=1$ and $q=p$ as it violates condition (iii). Hence, for every $q\in T$ if $\p$ is a prime ideal dividing $C_q$ we must have that $q=f_\a(\p)$ and that $v_\p(C_q) =\delta_\a(\p)$. \newline 

We can now state precisely which primes can divide $C_q$ and to which power. 
\begin{itemize}
    \item \textit{Non-Wieferich prime factors.} Suppose that $\p$ is a base-$\a$ non-Wieferich prime, i.e.,  $\delta_\a(\p)=1$.  If $\p\mid C_q$ for some $q\in T$, then $v_\p(C_q)=1$.
    \item \textit{Super-Wieferich prime factors.} If $\p$ is a super-Wieferich prime factor of  $C_q$  for some $q\in T$, then $f_\a(\p)=q$. This violates condition (ii) for membership in $T$. Therefore, no super-Wieferich prime can divide $C_q$. 
    \item \textit{Wieferich prime factors.} Suppose $\p$ is a Wieferich prime but not a super-Wieferich prime, i.e., $\delta_\a(\p)=2$. If $\p\mid C_q$ for some $q\in T$,  then $v_\p(C_q) =2$.
\end{itemize}

We now proceed by contradiction. Suppose there are finitely many non-Wieferich primes in $T$, say $\p_1, \ldots, \p_r$. For each prime $q\in T$ we will consider the factorization of $C_q$. From the above, for each $q\in T$ we obtain a factorization of the form
\[
C_q = \p_1^{\eps_1}\cdots \p_r^{\eps_r} w^2
\]
where $\eps_i\in \{0,1\}$ and $w$ is a product of distinct base-$\a$ Wieferich primes. \newline

Let $S$ be a multiplicative set so that $\O_{K,S}$ has class number one. The prime ideals of $\O_{K,S}$ are of the form $S^{-1}\p$ for prime ideals $\p\in \O_K$ that do not meet $S$, i.e., $\p\cap S=\emptyset$. If $\p\cap S$, then $\p$ becomes the unit ideal in $\O_{K,S}$. We will denote by $s$, the unit rank of $\O_{K,S}$. Localizing the ideal $C_q$ at $S$ gives:
\[
S^{-1}C_q = (S^{-1}\p_1)^{\eps_1}\cdots (S^{-1}\p_r)^{\eps_r} (S^{-1}w)^2. 
\]
By assumption there are finitely many non-Wieferich primes and hence there are infinitely many non-Wieferich primes. As such, there are only finitely many primes $q$ for which $S^{-1}C_q$ is the unit ideal in $\O_{K,S}$. If $\p_i$ is a non-Wieferich prime that does not meet $S$, then $S^{-1}\p_i$ is a prime ideal in $\O_{K,S}$, and hence principal. Write $\pi_i$ for the prime element that generates each $S^{-1}\p_i$. Without loss of generality, we assume that there are $0\leq m \leq r$ such $\pi_i$. Assume that $N$ is positive integer such that if $q\in T$ and $q\geq N$ then  $S^{-1}C_q$ is not the unit ideal. Then for each $q\geq N$ we have the following equation in $\O_{K,S}$:
\[
\a^q-1 = (\a-1)\eta  \pi_1^{\eps_1}\cdots \pi_m^{\eps_m}w'^2
\]
where $w'$ is a product of distinct localized Wieferich primes and $\eta$ is a unit of $ \O_{K,S}$.  We can write $\eta = \zeta \fun_1^{t_1}\cdots \fun_{s-1}^{t_{s-1}}$ where $\zeta$ is a root of unity of $\O_K$ and $\{\fun_i\}_{i=1}^{s-1}$ is a fundamental system of $S$-units. For each prime $q$ we write $q = 3q_i + \beta_i$ and for each $t_i$ we write $t_i = 2\ell_i +\gamma_i $ where $\beta_i \in \{0,1,2\}$ and $\gamma_i\in \{0,1\}$. We consider the finite family of elliptic curves defined over $K$ by

\[
(\a-1)AY^2 = BX^3-1
\]
where
\begin{align*}
    A = \zeta \fun_1^{\gamma_i}\cdots \fun_{s-1}^{\gamma_i}\pi_1^{\eps_1}\cdots \pi_m^{\eps_m}\,\,\, &\text{ and }\,\,\, B = \a^{\beta_i}.
\end{align*}

Denoting by $W_K$ the set of the roots of unity in $\O_K$, there are at most $\#W_K \cdot 2^{m+ s-1}\cdot 3 $ many curves in this family. By construction each prime $q\in T$ produces a unique $\O_{K,S}$ solution to one of these curves. However, this contradicts Siegel's theorem which implies that each of these curves has only finitely many solutions in $\O_{K,S}$ and hence the entire family of curves has only finitely many $\O_{K,S}$ solutions. Therefore, there must be infinitely many non-Wieferich primes. 
\end{proof}

\section{Proof of Theorem \ref{thm: thm 2}}

\begin{proof}[Proof of Theorem \ref{thm: thm 2}]

    We define two sets
    \begin{align*}
        S_1(x) &= \{\text{prime ideals }\p \subseteq \O_K : N(\p)\leq x \text{ and } \a^{\
        N(\p)-1}\not \equiv 1 \mod{\p^2}\}\\
        S_2(x) & = \{\text{prime ideals }\p \subseteq \O_K : \p \mid U_n \text{ for some $n$ and } N(U_n)\leq x\}. 
    \end{align*}
    By Lemma \ref{lem: non-wieferich criteria}, it is clear that $S_2(x) \subseteq S_1(x)$ so we will find a lower bound for $S_2(x)$. Theorem~\ref{thm: thm 1} and Theorem \ref{thm: 3} imply that $N(U_n) \to \infty$. As such, there is a positive integer $M$ (depending on $K, \eps,$ and possibly $\a$) such that $N(U_n)\geq N(\a-1)$ for all $n\geq M$. \newline

    By Lemma \ref{lem: upperbound for norm} we see that $N(\a^n-1)\leq x$ if  
    \[
        n\leq \frac{\log (\frac{x}{2})}{\log h(\a)}. 
    \]
    Noting that $\log(x/2) \geq \frac{1}{2}\log(x)$ for all $x\geq 4$, any positive integer $n\leq \frac{\log x}{2\log h(\a)}$ will satisfy $N(\a^n-1)\leq x$. Since $N(U_n) > N(U_1)$ for all $n\geq M$, there must be some prime ideal divisor of $U_n$ that does not divide $(\a-1)\O_K$. This observation combined with Lemma \ref{lem: gcd of ideals} implies that for each rational prime $q$ in the range
    \[
    M \leq q \leq \frac{\log x}{\log h(\a)}
    \]
    there is some prime ideal $\p$ such that $\p\mid U_q$ and $\p\nmid U_\ell$ for any other rational prime $\ell\neq q$. Hence, 
    \[
    S_2(x) \supseteq \left\{\text{rational primes $p$} : M\leq p\leq \frac{\log x}{2\log h(\a)}\right\}
    \]
    where $\pi(x)$ is the prime counting function. We deduce
    \[
    |S_1(x)|\geq |S_2(x)| \geq \pi\left(\frac{\log x}{2\log h(\a)} \right) - \pi(M). 
    \]
    Using any sort of Chebyshev-like lower bound for the prime counting function, $\pi(x)$, and noting that $\pi(M)=O(1)$, we conclude
    \[
    |S_1(x)| \gg_{K, \a, \eps} \frac{\log x}{\log\log x}
    \]
    as $x\to \infty$. 

\end{proof}
\section*{Acknowledgments} 
The research of the first author was partially supported by an Ontario Graduate Scholarship. Research of the second author was partially supported by an NSERC Discovery grant.   We thank the anonymous referee for their helpful comments.

\printbibliography
\vspace*{1cm}
{\small \textsc{Nic Fellini\newline 
Department of Mathematics and Statistics\newline 
Queen's University\newline  Kingston, ON, Canada  K7L 3N8}\newline 
\texttt{n.fellini@queensu.ca}}

\vskip 10pt
{\small \textsc{M. Ram Murty \newline 
Department of Mathematics and Statistics\newline 
Queen's University\newline  Kingston, ON, Canada  K7L 3N8}\newline 
{\tt murty@queensu.ca}} 
\end{document}